\documentclass[]{amsart}
\usepackage{graphicx}
\usepackage{amssymb,array, amscd,slashed} 
\usepackage{amsmath}
\newtheorem{Theorem}{Theorem}[section]
\newtheorem{Lemma}[Theorem]{Lemma}

\newtheorem{Definition}{Definition}
\newtheorem{Remark}[Theorem]{Remark} 
\numberwithin{equation}{section}

\newcommand{\R}{\mathbb R}
\newcommand{\C}{\mathbb C}
\newcommand{\Z}{\mathbb Z}
\newcommand{\ps}{\sc \small PS}
\newcommand{\SU}{{\rm SU}_2}
\newcommand{\su}{\mathfrak {su}_2}
\newcommand{\LSU}{\Lambda {\rm SU}_2}

\newcommand{\SL}{{\rm SL}_2 \mathbb C}
\newcommand{\LSUP}{\Lambda^+ {\rm SU}_2}
\newcommand{\LSUN}{\Lambda^- {\rm SU}_2}
\newcommand{\LSUPN}{\Lambda^+_{*} {\rm SU}_2}
\newcommand{\LSUNN}{\Lambda^-_{*} {\rm SU}_2}
\newcommand{\LSUPM}{\Lambda^{\pm} {\rm SU}_2}

\newcommand{\id}{\operatorname{Id}}
\newcommand{\ad}{\operatorname{Ad}}
\newcommand{\di}{\operatorname{diag}}

\newcommand{\etn}{\eta_n}
\newcommand{\etm}{\eta_m}

\renewcommand{\l}{\lambda}
\newcommand{\E}{\mathbb E}
\newcommand{\D}{\mathbb D}
\newcommand{\eq}[1]{\begin{align*}#1 \end{align*}}

\begin{document}
\title{A construction method for discrete constant negative Gaussian curvature surfaces}
 \author[S.-P.~Kobayashi]{Shimpei Kobayashi}
 \address{Department of Mathematics, Hokkaido University, 
 Sapporo, 060-0810, Japan}
 \email{shimpei@math.sci.hokudai.ac.jp}

\begin{abstract}
 This article is an application of the author's paper \cite{DKDPW} 
 about a construction method 
 for discrete constant negative Gaussian curvature surfaces, the nonlinear d'Alembert 
 formula. The heart of this formula is the Birkhoff decomposition, and 
 we give a simple algorithm for the Birkhoff decomposition 
 in Lemma \ref{lem:factorize}. As an application, we 
 draw figures of discrete constant negative Gaussian curvature 
 surfaces given by this method (Figures \ref{fig:1} and \ref{fig:2}).
\end{abstract}

\maketitle

\section{Introduction}
 The study of smooth constant negative Gaussian curvature surfaces ({\ps} surfaces\footnote{A constant negative Gaussian curvature surface is sometimes called a pseudospherical surface, thus we use ``{\ps}'' for the shortened name.} in 
 this article) is a classical subject of differential geometry. 
 It is known that the Gauss-Codazzi equations (nonlinear partial differential 
 equations) for a {\ps} surface become a famous integrable system, {\it sine-Gordon} 
 equation:
 \eq{\partial_{y} \partial_x u- \sin u =0.}
 One of the prominent features of integrable systems is that they can be obtained by 
 compatibility conditions for certain linear partial differential equations, the 
 so-called {\it Lax pairs}. Moreover, the Lax pair contains an additional parameter, 
 the {\it spectral parameter}, and it is a fundamental tool to 
 study integrable systems.
 On a {\ps} surface, 
 the spectral parameter induces a family of {\ps} surfaces, which will be called 
 the {\it associated family}, and the Lax pair is a family of moving frames 
 (Darboux frames) and it will be called the {\it extended frame} of 
 a {\ps} surface. The extended frame can be thought as an element of 
 the set of maps from the unit circle $S^1$ in the complex plane into a Lie group, the {\it loop group}, see Appendix for the definition. 

 In \cite{Krichever, Toda}, it was shown that loop group decompositions 
 (Birkhoff decompositions, see Theorem \ref{thm:Birkhoff}) 
 of the extended frame $F$ of a {\ps} surface 
 induced a pair of $1$-forms $(\xi_+, \xi_-)$, 
 that is, $F=F_+ F_-=G_- G_+$ with $\xi_+ = F_+^{-1} d F_+$ and 
 $\xi_- = G_-^{-1} d G_-$. Then it was proved that 
 $\xi_+$ and $\xi_-$ depended only on $x$ and $y$, respectively.
 Conversely it was shown that solving the pair of ordinary differential
 equations $d F_+ = F_+ \xi_+$ and $d G_+ = G_+ \xi_-$ and using 
 the loop group decomposition, the extended frame could be recovered.
 This construction is called the {\it nonlinear d'Alembert formula} for 
 {\ps} surfaces. 
 
 On the one hand a discrete analogue of smooth {\ps} surfaces 
 was defined in \cite{BP} and 
 the nonlinear d'Alembert formula for discrete {\ps} surfaces 
 was recently shown in \cite{DKDPW}.
 In this article we first review basic results for smooth/discrete 
 {\ps} surfaces and the nonlinear d'Alembert formula for smooth/discrete {\ps}
 surfaces according to \cite{DIS, BP, DKDPW}. 
 The heart of the formula is the Birkhoff decomposition. 
 We next give a simple algorithm (Lemma \ref{lem:factorize}) 
 for the Birkhoff decomposition. As an application, we finally
 draw figures of discrete pseudospherical surfaces given by 
 this method (Figures \ref{fig:1} and \ref{fig:2}).
\section{Preliminaries}\label{sc:Pre}
 We briefly recall basic notation and results about 
 smooth and discrete {\ps} surfaces in the Euclidean three space
 $\E^3$, that is $\R^3$ with the standard inner product $\langle\cdot, \cdot\rangle$, see for examples \cite{MS, BP, Toda, DIS, Krichever}.
 Moreover, we recall the nonlinear d'Alembert formula for discrete {\ps} surfaces 
 \cite{DKDPW}. 
 
\subsection{Pseudospherical surfaces}\label{subsc:pssurf}
 We first identify $\E^3$ with the Lie algebra 
 of the special unitary group $\SU$, which will be denoted by $\su$: 
\begin{equation}\label{eq:identification}
  {}^t(x, y, z) \in \E^3 \longleftrightarrow
 \frac{i}{2}x \sigma_1 - \frac{i}{2}  y \sigma_2 + \frac{i}{2} 
 z \sigma_3 \in \su, 
\end{equation}
 where $\sigma_j \;(j=1, 2, 3)$ are the Pauli matrices as follows:
\eq{
 \sigma_1=
\begin{pmatrix}
0 & 1 \\
1 & 0
\end{pmatrix},\;\;
\sigma_2=
\begin{pmatrix}
0 & -i \\
i & 0
\end{pmatrix}
\;\;\mbox{and}\;\;
 \sigma_3 = 
\begin{pmatrix}
1 & 0 \\
0 & -1
\end{pmatrix}.
}
 Note that the inner product of $\E^3$ can be computed as  
 $\langle \mathbf{x}, \mathbf{y} \rangle = -2 \operatorname{trace} (XY)$, where 
 $\mathbf{x}, \mathbf{y} \in \E^3$ and $X,Y \in \mathfrak su_2$ are the 
 corresponding matrices in \eqref{eq:identification}.
 Let $f$ be a {\ps} surface in $\E^3$ with Gaussian curvature $K =-1$. 
 It is known that there exist the {\it Chebyshev coordinates} 
 $(x, y) \in \R^2$ for $f$, that is, they are
 asymptotic coordinates normalized by $|f_x| = |f_y| =1$. 
 Here the subscripts $x$ and $y$ denote the 
 $x$- and $y$-derivatives $\partial_x$ and $\partial_y$, 
 respectively.  Then the first and second fundamental forms for $f$ can be
 computed as 
\eq{
 {\rm I}  =dx^2 + 2 \cos u\; dx dy + dy^2, \;\;\;
 {\rm I\!I}  = 2 \sin u \; dx dy,
}
 where $0< u< \pi/2$ is the angle between two asymptotic lines.
 Let 
\eq{
e_1 = \frac{1}{2} \sec (u/2) (f_x +f_y), \;\; e_2 = \frac{1}{2} \csc (u/2) (f_x -f_y)
 \;\;\mbox{and}\;\; e_3 = e_1 \times e_3} 
 be the Darboux frame rotating on the tangent plane clockwise angle $u$. 
 Note that it is easy to see that $\{e_1, e_2, e_3\}$ is an orthonormal basis of 
 $\E^3$.
 Under the identification \eqref{eq:identification}, 
 $\{ - \frac{i}{2} \sigma_1, - \frac{i}{2} \sigma_2, 
 - \frac{i}{2} \sigma_3\}$ is an orthonormal basis of $\E^3$, and for 
 a given $F \in \SU$ and $x \in \su$, ${\rm Ad} (F) (x) (:= F x F^{-1})$ denotes the rotation of 
 $x$. Thus there exists a $\widetilde F$ taking values in $\SU$ such that
\begin{equation}\label{eq:Darbouxframe}
 e_1 =  - \frac{i}{2} \widetilde F\sigma_1\widetilde F^{-1}, \;\;
 e_2 =  -\frac{i}{2} \widetilde F\sigma_2\widetilde F^{-1} \;\;
 \mbox{and}\;\;
 e_3 = - \frac{i}{2} \widetilde F\sigma_3\widetilde F^{-1}.
\end{equation}
 Without loss of generality, at some base point 
 $(x_*, y_*) \in \R^2$, we have $\tilde F(x_*, y_*)=\id$.
 Then there exists a family of frames $F$ parametrized by $\l \in \R_+:=
\{ r \in \R \;|\; r>0\}$ 
 satisfying  the following system of partial differential equations, see \cite{DIS} in 
 detail:
\begin{equation}\label{eq:Laxpair}
 F_x = F U \;\;\mbox{and} \;\;F_y = F V,
\end{equation}
 where 
\begin{equation}\label{eq:movingUV}
 U= 
\frac{i}{2}
 \begin{pmatrix}
- u_x   &  \l \\
  \l    & u_x
 \end{pmatrix}, \;\;
 V= -\frac{i}{2} 
 \begin{pmatrix}
 0  &  \l^{-1} e^{i u} \\
 \l^{-1} e^{-i u}     & 0
 \end{pmatrix}.
\end{equation}
 The parameter $\lambda \in \R_{+}$ will be called the {\it spectral parameter}.
 We choose $F$ such that 
\eq{
F|_{\l=1} = \widetilde F \;\;\mbox{and}\;\;
F|_{(x_*, y_*)}  = \id.
}
 The compatibility condition of the system in \eqref{eq:Laxpair}, 
 that is $U_y -V_x +[V, U] =0$, 
 becomes a version of the sine-Gordon  equation:
\begin{equation}\label{eq:sine-Gordon}
 u_{x y} - \sin u =0.
\end{equation}
 It turns out that the sine-Gordon equation is the Gauss-Codazzi equations for 
 {\ps} surfaces. Thus from the fundamental theorem of surface theory
 there exists a family of {\ps} surfaces 
 parametrized by the spectral parameter $\l \in \R_+$.
 Then the family of frames $F$ will be called the {\it extended frame} for $f$. 
 
 From the extended frame $F$, a family of {\ps} surfaces $f^{\lambda},\;
 (\l \in \R_+)$ 
 is given by the so-called {\it Sym formula}, \cite{Sym}:
 \begin{equation}\label{eq:Symformula}
 f^{\lambda} = \lambda \left.  \frac{\partial F}{\partial \lambda} 
 F^{-1}\right|_{\l  \in \R_{+}}.
 \end{equation}
 The immersion $f^{\l}|_{\l=1}$ is the original {\ps} surface $f$
 up to rigid motion. The one-parameter family $\{f^{\l}\}_{\l \in \R_{+}}$ 
 will be called the {\it associated family} of $f$.

\subsection{Nonlinear d'Alembert formula}
 Firstly, we note that the extended frame $F$ of a {\ps} surface $f$ is an 
 element of the {\it loop group} for $\SU$, that is, it is a set of 
 smooth maps from $S^1$ into $\SU$, see Appendix 
 for the definition. 
 In fact by \eqref{eq:movingUV} the extended frame is defined on 
 $\mathbb C^{\times} = \mathbb C \setminus \{0\}$ and it can be 
 thought as an element in the loop group of $\SU$.
 Then the loop group becomes a Banach Lie group with 
 suitable topology, which is 
 an infinite-dimensional Lie group and thus it will be called the loop group. 
 Then the Birkhoff decomposition of the loop group 
 is fundamental, 
 which will be now explained. The loop group of $\SU$ will be denoted by $\LSU$ and 
 we consider two subgroups $\LSUP$ and $\LSUN$ of $\LSU$ as sets of 
 maps which can be extended inside 
 the unit disk and outside the unit disk, respectively. 
 In other words, maps $F \in \LSU$, $F_+ \in \LSUP$ and $F_- \in 
 \LSUN$ have the following Fourier expansions:
\eq{
  F = \sum_{j =- \infty}^{\infty} F_j \l^{j}, \;\; 
  F_+ = \sum_{j =0}^{\infty} F^+_j \l^{j} \;\;\;\mbox{and} \;\;\;
  F_- = \sum_{j =- \infty}^{0} F^-_j \l^{j}.
}
 Then we consider the following problem: for a given map $F \in \LSU$, 
 does there exist $F_{\pm}$ or $G_{\pm}$ taking values in $\LSUPM$ such that 
\eq{
 F = F_+ F_- \;\;\mbox{or}\;\;  F= G_- G_+
}
 holds? The Birkhoff decomposition theorem assures that this decomposition 
 always holds in case of the loop group of $\SU$, 
 see Theorem \ref{thm:Birkhoff} in detail.
 By using the Birkhoff decomposition theorem, we give a construction method for {\ps} 
 surfaces, the so-called {\it the nonlinear d'Alembert formula}.

 From now on, for simplicity, we assume that the base point is $(x_*, y_*) =(0, 0)$ and 
 the extended frame $F$ at the base point is identity:
\eq{
 F(0, 0, \lambda) = \id.
}
 The nonlinear d'Alembert formula for smooth {\ps} surfaces
 is summarized as follows, \cite{Krichever, Toda, DIS}. 
\begin{Theorem}[\cite{Toda, DIS}]\label{thm:dAlembert}
 Let $F$ be the extended frame for a {\ps} surface $f$ in $\E^3$.
 Moreover, let $F= F_{+} F_{-} = G_{-} G_{+}$ be the Birkhoff 
 decompositions given in Theorem \ref{thm:Birkhoff}, respectively.
 Then $F_{+}$ and $G_{-}$ do not depend on $y$ and $x$, respectively, and 
 the Maurer-Cartan forms of $F_{+}$ and $G_{-}$ are given as follows$:$
\begin{equation}\label{eq:potential} 
\left\{
\begin{array}{l}
 \xi_+ = \displaystyle F_{+}^{-1} d F_{+} 
 =\frac{i}{2} \l 
\begin{pmatrix}
 0 & e^{-i \alpha(x)} \\
 e^{i \alpha(x)} & 0
\end{pmatrix} dx, \\[0.3cm]
\xi_- = G_{-}^{-1} d G_{-} 
 =\displaystyle
 -\frac{i}{2}\l^{-1} 
\begin{pmatrix}
 0 & e^{i \beta(y)} \\
 e^{-i \beta(y)} & 0
 \end{pmatrix}dy,
\end{array}
\right.
\end{equation}
 where, using the angle function $u(x, y)$, $\alpha$ and $\beta$ are given by
\eq{
 \alpha(x) = u(x, 0)-u(0, 0)\;\;\mbox{and}\;\;\beta(y) 
 = u(0, y).
}
 Conversely, let $\xi_{\pm}$ be a pair of $1$-forms defined in \eqref{eq:potential}
 with functions $\alpha(x)$ and $\beta(y)$ satisfying $\alpha(0)=0$.
 Moreover, let $F_{+}$ and $G_{-}$ be solutions of 
 the pair of the following ordinary differential equations$:$
 \begin{equation*}
\left\{
\begin{array}{l}
 d F_{+} = F_{+} \xi_{+}, \\
 d G_{-} = G_{-} \xi_{-},
\end{array}
\right.
\end{equation*}
 with $F_{+} (x =0, \l) = G_{-}(y =0, \l) = \id$.
 Moreover let $D = \di (e^{-\frac{i}{2} \alpha}, e^{\frac{i}{2} \alpha})$ and 
 decompose $(F_{+}D)^{-1} G_{-}$ by the Birkhoff decomposition 
 in Theorem \ref{thm:Birkhoff}$:$
\eq{
(F_{+}D)^{-1} G_{-} = V_{-} V_{+}^{-1},
}
 where $V_{-} \in \LSUNN$ and $V_{+} \in \LSUP$.
 Then $F = G_{-} V_{+} = F_{+} DV_{-}$ is the extended frame of some {\ps} surface in $\E^3$.
\end{Theorem}
\begin{Definition}
 The pair of $1$-forms $(\xi_+, \xi_-)$ in \eqref{eq:potential}
 will be called the {\it pair of normalized potentials}.
\end{Definition}

 \begin{Definition}
 In \cite{DIS}, it was shown that the extended frames of {\ps} surfaces 
 can be also constructed from the following pair of $1$-forms:
 \begin{equation}\label{eq:genepot}
\eta^x = \sum_{j = - \infty}^{1} \eta^x_{j} \lambda^{j} dx
 \;\;\mbox{and}\;\;
\eta^y = \sum_{j = - 1}^{\infty} \eta^y_{ j} \lambda^{j} dy, 
\end{equation}
 where $\eta_j^x$ and  $\eta_j^y$ take values in $\su$, 
 and each entry of $\eta_j^x$ (resp. $\eta_j^y$) is smooth on $x$ (resp. $y$), 
 and $\det \eta^x_1 \neq 0,  \det \eta^y_{-1} \neq 0$. Moreover $\eta_j^x$ 
 and $\eta_j^y$ are diagonal (resp. off-diagonal) if $j$ is even (resp. odd).
 This pair of $1$-forms $(\eta^x, \eta^y)$ is a generalization 
 of the normalized potentials $(\xi_+, \xi_-)$ in \eqref{eq:potential} 
 and will be called the {\it pair of generalized potentials}, see also \cite{BIK}.
\end{Definition}

\subsection{Discrete pseudospherical surfaces}
 Discrete {\ps} surfaces were first defined in \cite{BP}. 
 Instead of the smooth coordinates 
 $(x, y) \in \R^2$, 
 we use the quadrilateral lattice $(n, m) \in \Z^2$, that is,
 all functions depend on the lattice $(n, m) \in \Z^2$.
 The subscripts $1$ and $2$ (resp. $\bar 1$ and $\bar 2$)
 denote the forward (resp. backward) lattice points with 
 respect to $n$ and $m$: For a map $s(n, m)$ of the lattice 
 $(n, m) \in \Z^2$, we define $s_1, s_2, s_{\bar 1}$ and $s_{\bar 2}$  by
\eq{
 s_{1} = s(n+ 1, m), \; s_{\bar 1} = s(n- 1, m),\;
 s_{2} = s(n, m+ 1) \;\;\mbox{and}\;\; s_{\bar 2} = s(n, m-1).
}
 A discrete {\ps} surface $f : \Z^2 \to \E^3$ 
 was defined by the following two conditions:
 \begin{enumerate}
\item For each point $f \in \E^3$, there is a plane $P$ such that 
\eq{
 f, \; f_1, \; f_{\bar 1}, \; f_2, \; f_{\bar 2} \in P.
} 
\item The length of the opposite edge of an elementary quadrilateral are equal:
\eq{
 |f_1 - f| = |f_{1 2} -f_2| = a(n) \neq 0, \;\;
 |f_2 - f| = |f_{1 2} -f_1| = b(m) \neq 0.
} 
\end{enumerate}
 Then the {\it discrete extended frame} $F$, which takes values in $\LSU$, 
 of a discrete {\ps} surface $f$
 can be defined by  the following partial difference system, 
 see \cite[Section 3.2]{BP:surv} and 
 \cite{BP}:
\begin{equation}\label{eq:DLaxpair}
 F_1 = F U \;\;\mbox{and} \;\;F_2 = F V,
\end{equation}
 where 
\begin{equation}\label{eq:DmovingUV}
\left\{
\begin{array}{l}
\displaystyle 
 U = 
\frac{1}{\Delta_{+}}
 \begin{pmatrix}
 e^{-\frac{i}{2} (u_1 -u)}  &  \frac{i}{2}p \lambda  \\
 \frac{i}{2} p \lambda     &  e^{\frac{i}{2}(u_1 -u)}
 \end{pmatrix},  \;\;
 \\[0.2cm]

\displaystyle 
 V= \frac{1}{\Delta_{-}} 
 \begin{pmatrix}
 1  &  -\frac{i}{2} q e^{\frac{i}{2}(u_2 + u)} \l^{-1}\\
- \frac{i}{2} q e^{-\frac{i}{2} (u_2 + u)} \l^{-1}& 1
 \end{pmatrix},
\end{array}
\right.
\end{equation}
 with $\Delta_+ = \sqrt{1 + (p/2)^2 \l^2}$ and 
 $\Delta_- = \sqrt{1 + (q/2)^2 \l^{-2}}$. Here 
 $u$ is a real function depending on both $n$ and $m$, and 
 $p\neq 0$ and $q\neq 0$ are real functions depending only on $n$ and $m$, 
 respectively: 
\eq{
 u = u(n, m), \;\;
 p = p(n) \;\;\mbox{and}\;\;
 q = q(m).
 }
 The compatibility condition of the system in \eqref{eq:DLaxpair},
 that is $V U_2 = UV_1$, 
 gives the so-called {\it discrete sine-Gordon equation}:
\begin{equation}\label{eq:dsineGordon}
 \sin \left( \frac{u_{12}-u_1-u_2+u}{4}\right)
 = \frac{p q}{4} \sin \left( \frac{u_{12}+u_1 + u_2+u}{4}\right).
\end{equation}
 The equation \eqref{eq:dsineGordon} was first found by Hirota in \cite{Hirota}
 and also called the {\it Hirota equation}.
\begin{Remark}\label{rm:length}
 Strictly speaking, the lengths of the edges for a discrete {\ps} surface should be small
 (less than $1$).
 If the length is big (greater than or equal to $1$), 
 then the compatibility condition $VU_2 = UV_1$ gives 
 a discrete analogue of mKdV equation, see \cite{IKMO} in detail.  
 Since this restriction is fundamental for the discrete nonlinear d'Alembert formula, 
 we assume the conditions in \eqref{eq:conditions}.
\end{Remark}
 Then the discrete {\ps} surface $f$ can be given by the so-called Sym formula, 
 \cite{BP}:
 \begin{equation}\label{eq:dSymformula}
 f^{\lambda} = \lambda \left.  \frac{\partial F}{\partial \lambda} 
 F^{-1}\right|_{\l  \in \R_{+}}.
 \end{equation}
 The original discrete {\ps} surface $f$ and 
 $f^{\lambda}|_{\lambda =1}$ are the same surface up to rigid motion. 
 It is easy to see that the map $f^{\lambda}$ 
 defined in \eqref{eq:dSymformula} satisfies two properties of a 
 discrete {\ps} surface and $\{f^{\l}\}_{\lambda \in \R_{+}}$ gives 
 a family of discrete {\ps} surfaces, see \cite[Theorem 3]{BP:surv}.
\subsection{Nonlinear d'Alembert formula 
 for discrete {\ps} surfaces}\label{subsc:DDAlembert}
 In this subsection we assume that the base point is $(n_*, m_*)=(0, 0)$
 and the discrete extended frame $F$ at the base point is identity:
\eq{
 F(0, 0, \lambda) = \id.
}
 Moreover, we also assume that the functions $p$ and $q$ in \eqref{eq:DmovingUV}
 satisfy the inequalities,
\begin{equation}\label{eq:conditions}
 0 < \left|\frac{p}{2}\right|<1 \;\;\mbox{and}\;\;0<\left|\frac{q}{2}\right|<1,
\end{equation}
 see Remark \ref{rm:length}.
 Then the discrete nonlinear d'Alembert formula can be summarized as follows.

 \begin{Theorem}[\cite{DKDPW}]\label{thm:potential}
 Let $f$ be a discrete {\ps} surface and $F$ the corresponding 
 discrete extended frame. Decompose $F$ according 
 to the Birkhoff decomposition in Theorem \ref{thm:Birkhoff}$:$
\eq{
 F= F_{+} F_{-} = G_{-} G_{+},
}
 where $F_{+} \in \LSUPN, F_{-} \in \LSUN, G_- \in \LSUNN$ 
 and $G_+ \in \LSUP$.
 Then $F_{+}$ and $G_-$ do not depend on $m \in \Z$ and 
 $n \in \Z$, 
 respectively, and the discrete Maurer-Cartan forms 
 of $F_+$ and $G_-$ are given 
 as follows:
\begin{equation}\label{eq:discretepotentials}
\left\{
\begin{array}{l}
\displaystyle
 \xi_{+} = F_{+}^{-1} (F_{+})_1 = 
 \frac{1}{ \Delta_{+}}
\begin{pmatrix}
1 & \frac{i}{2} p e^{-i \alpha}\l \\
\frac{i}{2} p e^{i \alpha} \l &1 
\end{pmatrix}, \\[0.3cm]
\displaystyle
 \xi_{-} = G_{-}^{-1} (G_{-})_2 = 
 \frac{1}{ \Delta_-}
\begin{pmatrix}
1 &  -\frac{i}{2} q e^{i \beta}\l^{-1} \\
-\frac{i}{2} q e^{-i \beta}\l^{-1} &1
\end{pmatrix},
\end{array}
\right.
\end{equation}
 where $\Delta_{+} = \sqrt{1 + (p/2)^2 \l^2}$ 
 and $\Delta_{-} = \sqrt{1 + (q/2)^2 \l^{-2}}$, 
 the functions $p$ and $q$ are given in \eqref{eq:DmovingUV},
 and $\alpha$ and $\beta$ are functions of $n \in \Z$ and $m \in \Z$, 
 respectively.
 Moreover using the function $u(n, m)$ in \eqref{eq:DmovingUV}, 
 $\alpha(n) $ and $\beta(m)$ are given by
 \begin{equation}\label{eq:potfunctions}
\left\{
\begin{array}{l}
 \alpha(n) =\frac{1}{2} u(n+1, 0)+\frac{1}{2}u(n, 0)- u(0, 0) , \\[0.1cm]
 \beta(m) = \frac{1}{2}u(0,m+1)+\frac{1}{2}u(0, m).
\end{array}
\right.
 \end{equation}
 Conversely, 
 Let $\xi_{\pm}$ be a pair of matrices defined in  \eqref{eq:discretepotentials} 
 with arbitrary 
 functions $\alpha = \alpha(n),  \beta = \beta(m)$ with $\alpha (0) = 0$ 
 and $p = p(n), q = q(m)$ satisfying the conditions \eqref{eq:conditions}.
 Moreover, let $F_{+} =F_+(n, \l)$ and $G_{-} =G_-(m, \l) $ be the solutions 
 of the ordinary difference equations
\begin{align}\label{eq:solofFG}
 (F_{+})_{1} = F_{+} \xi_{+}\;\;\mbox{and}\;\;
 (G_{-})_{2} = G_{-} \xi_{-},
\end{align}
 with $F_{+}(n =0, \l) =G_{-}(m=0, \l)  = \id$ and set a matrix
 $D= \di (e^{\frac{i}{2} k}, e^{-\frac{i}{2} k}) \in {\rm U}_1$, 
 where $k(0) =0$ and $k(n) = 2 \sum_{j=0}^{n-1}(-1)^{j+n}\alpha(j)$ for $n\geqq 1$.
 Decompose $(F_{+}D)^{-1} G_{-}$ by the Birkhoff decomposition in 
 Theorem \ref{thm:Birkhoff}$:$
\begin{equation}\label{eq:Birkhoffforpot}
 (F_{+}D)^{-1} G_{-}  = V_{-} V_{+}^{-1},
\end{equation}
 where $V_{-} \in \LSUNN, V_{+} \in \LSUP$.
 Then $F = G_{-} V_{+} = F_{+} D V_{-}$ 
 is the discrete extended frame of some discrete {\ps} surface 
 in $\E^3$. Moreover the solution $u = u(n, m)$ of the discrete sine-Gordon 
 for the discrete {\ps} surface satisfies the relations in \eqref{eq:potfunctions}.
\end{Theorem}

\begin{Definition}
 The pair of matrices $(\xi_{-}, \xi_{+})$ given in \eqref{eq:discretepotentials}
 will be called the
 {\it pair of  discrete normalized potentials}.
\end{Definition}
 Similar to the smooth case, we generalize the pair of discrete normalized potentials:
\begin{Definition}\label{def:genpot}
 Let $(\xi_{-}, \xi_+)$ be a pair of discrete normalized potentials and let $\etm$ and 
 $\etn$ be
\begin{equation}\label{eq:generalizedpot}
\etn = 
P_-^{l}\xi_+ P_-^{r}, 
\;\;
\etm =
P_+^l \xi_- P_+^r.
\end{equation}
 Here we assume that $P_{\pm}^{\star}$ ($\star = l$ or $r$)
 take values in $\LSUPM$  
 and do not depend on $m$ and $n$, respectively, that is, 
 $P_-^\star = P_-^\star(n, \l)$ and $P_+^\star = P_+^\star(m, \l)$.
 Thus the $\etn$ and $\etm$ do not depend on $m$ and $n$, respectively:
\eq{
 \etn = \etn (n, \l), \;\;
 \etm = \etm (m, \l).
}
 The pair $(\etn, \etm)$ given in \eqref{eq:generalizedpot} will be called 
 the {\it pair of discrete generalized potentials}.
\end{Definition}

\begin{Remark}
 The pair of normalized potentials $(\xi_+, \xi_-)$ and 
 the corresponding pair of discrete generalized potentials $(\etn, \etm)$ 
 in \eqref{eq:generalizedpot} give in general different discrete {\ps} 
 surfaces.
\end{Remark}
\section{Algorithm for Birkhoff decomposition}\label{sc:Algandex}
 In this section, we give a simple algorithm performing the Birkhoff decomposition 
 used in Theorem \ref{thm:potential}.
 
 When one looks at the discrete extended frame $F$ defined 
 in \eqref{eq:DLaxpair}, 
 $F_+$ and $G_-$ defined in \eqref{eq:solofFG}, 
 one notices that they are given by products of two types of matrices:
\begin{equation}\label{eq:Epm}
 e_{+} = 
\frac{1}{\sqrt{1+ |a|^2 \l^{2}}}\begin{pmatrix} e^{i \theta}  & a \l \\ -\bar a\l & e^{-i \theta} \end{pmatrix}, \;
e_{-} = \frac{1}{\sqrt{1+ |b|^2 \l^{-2}}}\begin{pmatrix} e^{i \kappa}  & b \l^{-1} \\ -\bar b\l^{-1} & 
 e^{-i \kappa}\end{pmatrix},
\end{equation}
 where $\theta, \kappa \in \R, a, b \in \C \;\mbox{and}\; |a|, |b|<1$. 
 It is easy to see that $e_{\pm}$ take values in $\LSUPM$, 
 respectively. Two matrices $e_+$ and $e_-$ 
 do not commute in general, however, the following lemma holds.
\begin{Lemma}\label{lem:factorize}
 Let $e_{\pm}$ be matrices in \eqref{eq:Epm}.
 Then there exist matrices $\tilde e_{\pm}$ which take values in $\LSUPM$ 
 such that 
\eq{e_+ e_- = \tilde e_- \tilde e_+} 
 holds. In particular $\tilde e_{\pm}$ can be explicitly computed as follows$:$
\eq{
\tilde e_{+} = \frac{1}{\sqrt{1+ |\tilde a|^2 \l^2}}\begin{pmatrix} e^{i \tilde \theta} 
 & \tilde a \l 
 \\ -\bar {\tilde a}\l & e^{- i \tilde \theta} \end{pmatrix}\; \;\mbox{and}\;\;
\tilde e_{-} = \frac{1}{\sqrt{1+ |\tilde b|^2 \l^{-2}}}\begin{pmatrix} e^{i \tilde \kappa} & \tilde b \l^{-1} 
 \\ -\bar{\tilde b}\l^{-1} & e^{- i \tilde \kappa} \end{pmatrix},} 
 where $\tilde a, \tilde b, \tilde \theta$ and $\tilde k$ are explicitly chosen by 
 the following equations$:$
\eq{
 \tilde a = a e^{-i (\kappa + \tilde \kappa)}, \;
 \tilde b = b e^{i (\theta + \tilde \theta)} \;\;\mbox{and}\;\; 
 \tilde \theta + \tilde \kappa= \theta + \kappa + 2 \arg 
 (1- a \bar{b}e^{-i(\theta + \kappa)}).}
 Note that $\tilde e_{\pm}$ are not unique and 
 one can always choose $\tilde \theta = 0$ or $\tilde \kappa  =0$.
 \end{Lemma}
\begin{proof}
 It is just a consequence of a direct computation of 
 $e_+ e_-$ and $\tilde e_{-} \tilde e_{+}$, respectively.
\end{proof}
 Using Lemma \ref{lem:factorize} iteratively, we obtain the following algorithm 
 for the Birkhoff decomposition.
\begin{Theorem}
 Let $F$ be the discrete extended frame of a {\ps} surface. Moreover let $F_+, G_-$ and $D$ be the matrices defined in \eqref{eq:solofFG}.
 Then the Birkhoff decompositions for $F$ and $(F_+ D)^{-1}G_-$ 
 can be explicitly computed.
\end{Theorem}
 
 As an example of the above theorem, we draw figures (figures \ref{fig:1} and \ref{fig:2})
 of discrete {\ps} surfaces of revolution according to the 
 following potential, see also \cite[Section 3]{DKDPW}. 
 Let $\etn$ and $\etm$ be $\etn = \etm^{-1} = A_+ L A_-$   
 with 
\eq{
 A_{\pm} = 
\frac{1}{\Delta_{\pm}} 
\begin{pmatrix} 
1 & \pm \frac{i}{2} q \l^{\pm1} \\
\pm \frac{i}{2} q \l^{\pm1}& 1
\end{pmatrix}\;\;
\mbox{and}\;\;L= \di (e^{ic}, e^{-ic}), 
}
 where $\Delta_{\pm} = \sqrt{1 + (q/2)^2 \l^{\pm 2}}$, $q \;(0< |q/2|<1)$ 
 and $c = \pi \ell^{-1} \;(\ell \in \Z_+)$ are some constants. 
 We note that $(\eta_n, \eta_m)$ is a pair of discrete generalized potentials 
 in Definition \ref{def:genpot}.
 The pair of solutions for $((F_{n})_1, (G_m)_2) = (F_n, G_m)(\eta_n, \eta_m)$ 
 with $F_n(0)=G_m(0) = \id$ is explicitly given by 
\eq{
F_n  = (A_+ L A_-)^n, \;\;\;\mbox{and}\;\;\;
G_m  = (A_+ L A_-)^{-m}, 
}
 respectively. Then the Birkhoff decomposition
\eq{
F_n^{-1} G_m= V_- V_+^{-1}
}
 is given by using Lemma \ref{lem:factorize}.
 In fact, $F_n^{-1} G_m$ can be rephrased as 
\eq{
F_n^{-1} G_m = (A_+ L A_-)^{-(n+m)}= \overbrace{(B_- B_+) (B_-B_+) \cdots (B_-B_+)}^{n+m},
}
 where we set $B_+ = (A_+ L)^{-1}$ and $B_- = A_-^{-1}$. Then by using 
 Lemma \ref{lem:factorize}, there exist $B_{+, 1}$ and $B_{-, 1}$ such 
 that $B_+ B_- = B_{-,1} B_{+,1}$ holds. Thus we can compute $F_n^{-1} G_m$
 as follows:
\eq{
 F_n^{-1} G_m = B_- (B_+B_-) \cdots(B_+ B_-)B_+ =B_- (B_{-, 1} B_{+, 1}) \cdots 
 (B_{-, 1} B_{+, 1}) B_+.
}
 Next, we use recursively Lemma \ref{lem:factorize} for $B_{-, i} B_{+, i}\;\;
  (i=1, 2, \ldots, n+m-1)$, that is, 
 there exist  $B_{+, i+1}$ and $B_{-, i+1}$ such that $B_{+, i} B_{-, i} = B_{-,i+1} B_{+, i+1}$  holds. Finally, $F_n^{-1} G_m$ can be computed as 
\eq{
 F_n^{-1} G_m =(B_- B_{-, 1} \cdots B_{-, n+m-1}) \cdot (B_{+, n+m-1} \cdots B_{+, 1} B_+).
}
 Note that since $B_-$ takes values in $\LSUNN$, 
 all $B_{-, i}$ also take values in $\LSUNN$.
 Thus $V_-= B_- B_{-, 1} \cdots B_{-, n+m-1}$ and $V_+^{-1} = B_{+, n+m-1} \cdots B_{+, 1} B_+$. 
 Then we do not need to compute the diagonal matrix $D$ as in Theorem \ref{thm:potential} 
 for drawing figures, since any $\l$-independent diagonal term goes away in
 Sym formula \eqref{eq:dSymformula}, that is, we can use $F_n  V_- = G_m V_+$ 
 in stead of the discrete extended frame $F$ which is given by 
 $F= F_n  V_- D = F_n  V_- D$.
\begin{figure}
\centering
\includegraphics[scale=0.45]{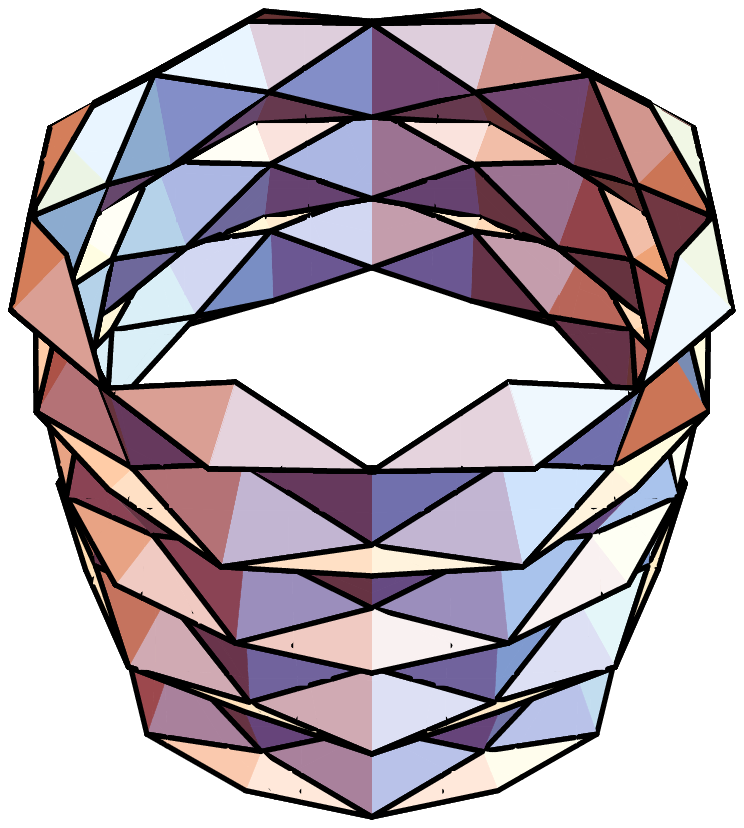}
\includegraphics[scale=0.4]{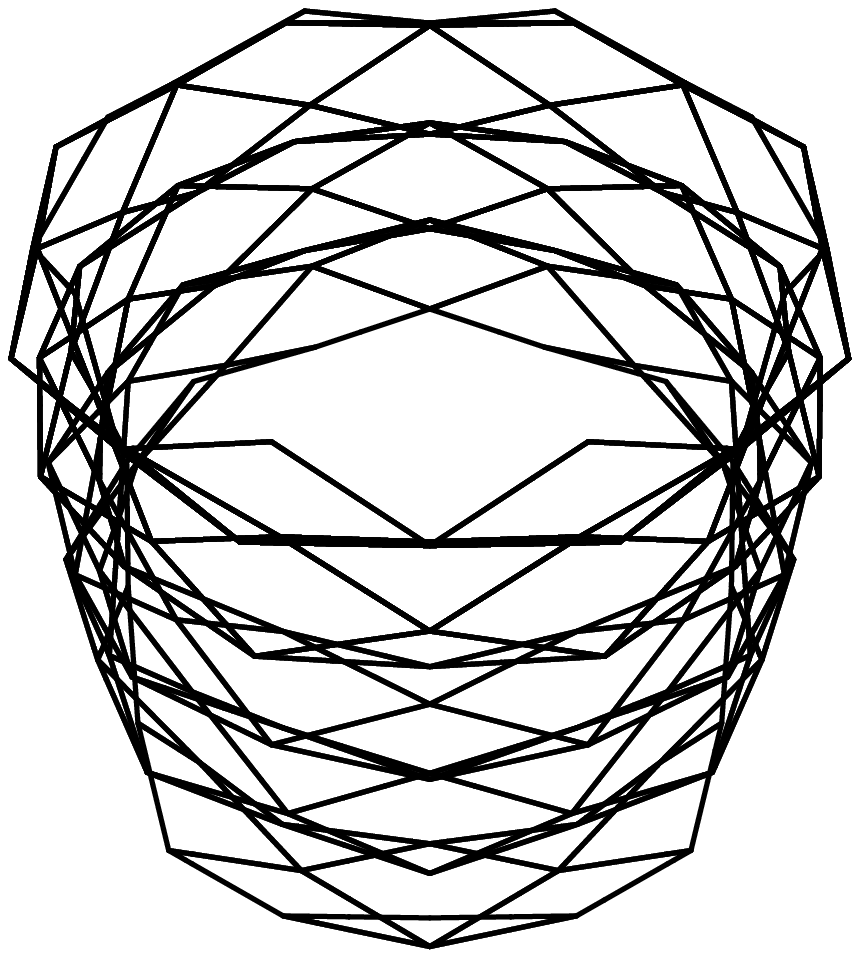}
\caption{A discrete {\ps} surface of revolution with
 the parameters $q =0.8$ and $c = \pi/8$. The right figure is a wired model of the 
 left one. The figures are made by using Wolfram Mathematica 10.}
\label{fig:1}
\end{figure}

\begin{figure}
\centering
\includegraphics[scale=0.35]{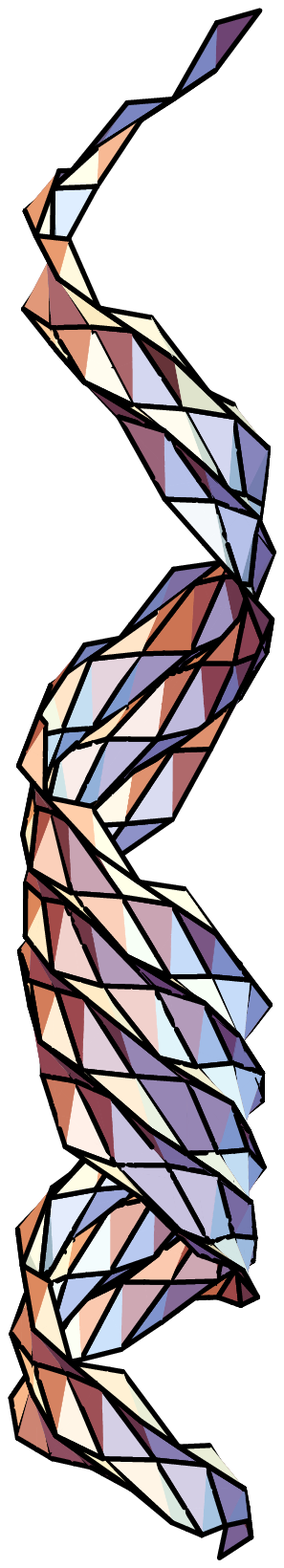}
\includegraphics[scale=0.32]{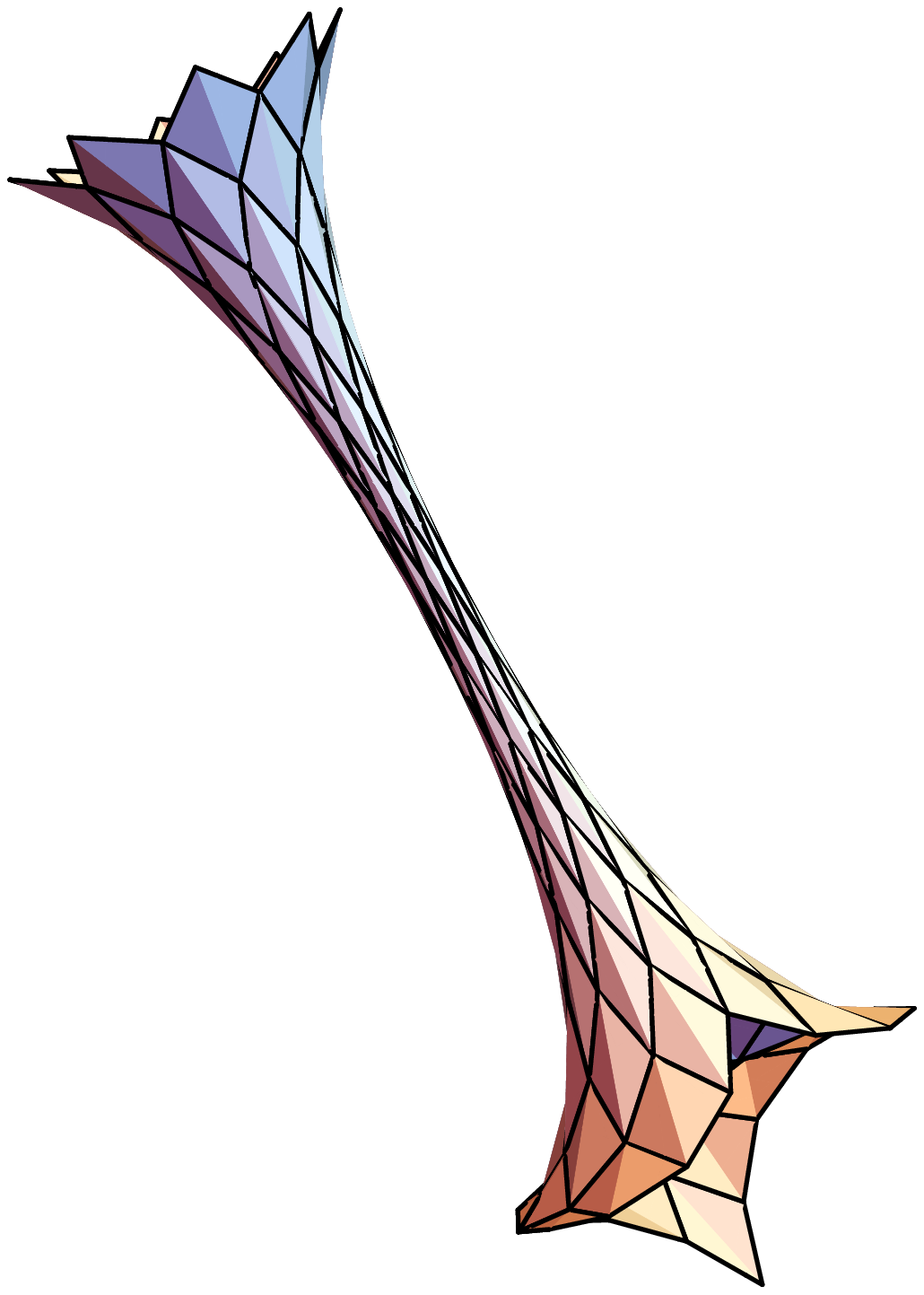}
\caption{
  (Left) a discrete {\ps} surface is given by the same data in Figure \ref{fig:1}
 and the spectral parameter has been chosen as $\l =1.5$ in Sym formula 
 \eqref{eq:Symformula}. The surface has a screw motion symmetry. 
 (Right) a discrete {\ps} surface of revolution with
 the parameters $q =0.4$ and $c = \pi/8$.
}
\label{fig:2}
\end{figure}

\section*{Appendix}\label{sc:loopgroups}
 In this appendix we give a definition of the loop group of $\SU$ and 
 its subgroups $\LSUPM$. Moreover theorem of the Birkhoff 
 decomposition will be stated.

 It is easy to see that 
 $F$ defined in \eqref{eq:Laxpair} together with the condition 
 $F|_{(x_*, y_*)}  = \id$ is an element in the {\it twisted $\SU$-loop group}:
\begin{equation}\label{eq:SUloop}
\LSU := \left\{ g : \R^{\times} \cup S^1 \to \SL\;\Big|\; 
\begin{array}{l}
\mbox{
 $g$ is smooth, $g(\lambda) = \left(\overline{g(\bar \lambda)}^{-1}
 \right)^{T}$}  \\
 \mbox{\;and $\sigma g(\lambda) = g(-\l)$}
 \end{array}\right\},
\end{equation} 
 where $\R^{\times} =\R \setminus \{0\}$, $\sigma X = \ad (\sigma_3) X = \sigma_3 X \sigma_3^{-1}, 
 (X \in \SL)$  is an involution on $\SL$. 
 In order to make the above 
 group a Banach Lie group, we restrict the occurring matrix coefficients to 
 the Wiener algebra $\mathcal A =\{ f(\l)= \sum_{n \in \Z} f_n \l^n : S^1
\rightarrow \mathbb C \;|\;  \sum_{n \in \Z} |f_n| < \infty \}$, where
 we denote the Fourier expansion of $f$ on $S^1$ by $f(\l)=\sum_{n \in \Z} f_n \l^n$.
 Then the Wiener algebra is a Banach algebra relative to the norm $\|f\| = \sum |f_n|$
 and the loop group $\LSU$ is a Banach Lie group, \cite{Gohberg}.

 Let $\D^{+}$ and $\D^{-}$ be the interior of the unit disk in 
 the complex plane and the union of the exterior of the unit disk in the complex plane 
 and infinity, respectively. 
 We first define two subgroups 
 of $\LSU$:
\begin{eqnarray}
 \LSUP  =\left\{ g \in\LSU \;|\;\mbox{$g$ can be analytically  extended 
 to $\D^+$} \right\}, \\
 \LSUN  =\left\{ g \in\LSU \;|\;\mbox{$g$ can analytically  be
 extended to $\D^{-}$} \right\}.
\end{eqnarray}
 Then $\LSUPN$ and $\LSUNN$ denote subgroups 
 of $\LSUP$ and $\LSUN$ normalized at $\l =0$ and $\l = \infty$, respectively:
\eq{
 \LSUPN =\left\{ g \in \LSUP \;|\; g(\l =0) = \id \right\},\\
 \LSUNN =\left\{ g \in \LSUN \;|\; g(\l =\infty) = \id \right\}. 
}
 The following decomposition theorem is fundamental.
\begin{Theorem}[Birkhoff decomposition, \cite{Gohberg, Brander}]
 \label{thm:Birkhoff}
 The multiplication maps 
\eq{
 \LSUPN \times \LSUN \to \LSU \;\;\mbox{and}\;\; 
 \LSUNN \times \LSUP \to \LSU
}
 are diffeomorphisms onto $\LSU$, respectively.
\end{Theorem}

\bibliographystyle{plain}

\begin{thebibliography}{10}

\bibitem{BP} A.~Bobenko and U.~Pinkall.
 \newblock Discrete surfaces with constant negative {G}aussian curvature
          and the {H}irota equation.
 \newblock {\em J. Differential Geom.},
 \textbf{43}, no. 3, 527 -- 611, 1996.

\bibitem{BP:surv} A.~Bobenko and U.~Pinkall.
 \newblock Discretization of surfaces and integrable systems. 
 \newblock {\em Discrete integrable geometry and physics (Vienna, 1996). Oxford Lecture Ser. Math. Appl.},
 \textbf{16}, 3 -- 58, 1999.

\bibitem{Brander} D.~Brander.
 \newblock Loop group decompositions in almost split real forms and
  applications to soliton theory and geometry.
 \newblock {\em J. Geom. Phys.}
 \textbf{58}, no. 12, 1792 -- 1800, 2008.

\bibitem{BIK} D.~Brander, J.~Inoguchi, S.-P.~Kobayashi.
 \newblock Constant {G}aussian curvature surfaces in the 3-sphere via loop groups.
 \newblock {\em Pacific J. Math.}, 
 \textbf{269}, no. 2, 281 -- 303, 2014.

\bibitem{DIS} J.~F.~Dorfmeister, T.~Ivey and I.~Sterling.
 \newblock Symmetric pseudospherical surfaces. {I}. {G}eneral theory.
 \newblock {\em Results Math.}, 
 \textbf{56}, no. 1 -- 4, 3--21, 2009.

\bibitem{Gohberg} I.~Ts.~Gohberg.
 \newblock The factorization problem in normed rings, functions of
 isometric and symmetric operators, and singular integral equations.
 \newblock {\em Russian Math. Surv.},
 \textbf{19}, 63 -- 114, 1964.

\bibitem{Hirota} R.~Hirota.
 \newblock Nonlinear partial difference equations. {III}. {D}iscrete
              sine-{G}ordon equation.
 \newblock {\em J. Phys. Soc. Japan},
 \textbf{43}, no. 6, 2079 -- 2086, 1977.

\bibitem{IKMO} J.~Inoguchi, K.~Kajiwara, N.~Matsuura and Y.~Ohta.
 \newblock Discrete m{K}d{V} and discrete sine-{G}ordon flows on discrete space curves.
 \newblock {\em J. Phys. A},
 \textbf{47}, no.23, 235022, 26pp, 2014.

\bibitem{DKDPW} S.-P.~Kobayashi.
 \newblock Nonlinear d'Alembert formula for discrete pseudospherical surfaces.
 \newblock {Preprint}, arXiv:1505.07189, 2015.

\bibitem{Krichever} I.~M.~Kri{\v{c}}ever.
 \newblock An analogue of the d'{A}lembert formula for the equations 
 of a principal chiral field and the sine-{G}ordon equation.
 \newblock {\em Dokl. Akad. Nauk SSSR},
 \textbf{253}, no. 2, 288 -- 292, 1980.

\bibitem{MS} M.~Melko and I.~Sterling.
 \newblock Application of soliton theory to the construction of pseudospherical surfaces in {${\bf R}^3$}.
 \newblock {\em Ann. Global Anal. Geom.}, 
 \textbf{11}, no. 1, 65 -- 107, 1993.

\bibitem{Sym} A.~Sym.
 \newblock Soliton surfaces and their applications 
 (soliton geometry from spectral problems).
 \newblock {\em Geometric aspects of the {E}instein equations and integrable
  systems ({S}cheveningen, 1984), Lecture Notes in Phys.}, 
 \textbf{239}, 154 -- 231, 1985.

\bibitem{Toda} M.~Toda.
 \newblock Weierstrass-type representation of weakly regular
 pseudospherical surfaces in {E}uclidean space.
 \newblock {\em Balkan J. Geom. Appl.}, 
 \textbf{7}, no. 2, 87 -- 136, 2002.
\end{thebibliography}
\def\cprime{$'$}

\end{document}